\documentclass[11pt]{amsart}
\usepackage[utf8]{inputenc}
\usepackage[margin=1.3in]{geometry}
\usepackage{times}
\usepackage{amsfonts}
\usepackage{amssymb}
\usepackage{amsthm}
\usepackage{mathtools}
\usepackage{mathrsfs}
\usepackage{caption}
\usepackage{subcaption}
\usepackage{bbm}
\usepackage[export]{adjustbox}
\usepackage{csquotes}
\usepackage{stmaryrd}
\usepackage[all]{xy}
\usepackage{tikz-cd}
\usetikzlibrary{matrix}
\usepackage{graphicx} 
\usepackage{float}
\usepackage{epstopdf}
\usepackage[linktocpage]{hyperref}
\hypersetup{colorlinks=true,linkcolor=red,citecolor=blue,urlcolor=green}
\usepackage{color}
\definecolor{note}{rgb}{0,0,1}
\usepackage{enumitem}

\newtheorem{theorem}{Theorem}[section]

\newtheorem{lemma}[theorem]{Lemma}
\newtheorem{proposition}[theorem]{Proposition}
\theoremstyle{definition}
\newtheorem{definition}[theorem]{Definition}
\newtheorem{remark}[theorem]{Remark}
\newtheorem{example}[theorem]{Example}

\numberwithin{equation}{section}
\numberwithin{theorem}{section}

\newcommand{\op}{\operatorname}

\newcommand{\be}{\begin{enumerate}}
\newcommand{\ee}{\end{enumerate}}

\usepackage[english]{babel}
\usepackage{hyphenat}
\usepackage[backend=bibtex,style=alphabetic,maxalphanames=4,maxnames=4]{biblatex}

\renewbibmacro{in:}{}
\DeclareDelimFormat[bib,biblist]{nametitledelim}{\addcomma\space}
\DeclareFieldFormat*{title}{\mkbibitalic{#1}\addcomma}
\DeclareFieldFormat*{journaltitle}{#1}
\DeclareFieldFormat*{volume}{\mkbibbold{#1}}
\DeclareFieldFormat{pages}{#1}
\DeclareFieldFormat[misc]{date}{preprint {#1}}
\DeclareFieldFormat{mr}{%
  MR\addcolon\space
  \ifhyperref
    {\href{http://www.ams.org/mathscinet-getitem?mr=MR#1}{\nolinkurl{#1}}}
    {\nolinkurl{#1}}}

\AtEveryBibitem{
  \clearfield{url}
  \clearfield{number}
  \clearfield{doi}
  \clearfield{issn}
  \clearfield{isbn}
  \clearfield{eprintclass}
}
\AtEveryBibitem{\ifentrytype{book}{\clearfield{pages}}{}}

\bibliography{references}

\usepackage{fancyhdr}
\usepackage{todonotes}

\RequirePackage{tikz}
\usetikzlibrary{arrows.meta,calc,decorations.markings}

\title[A polynomial invariant of planar curves]{A polynomial invariant of planar curves}

\author{Zhiyun Cheng}
\address{School of Mathematical Sciences, Beijing Normal University; 
Laboratory of Mathematics and Complex Systems, Ministry of Education, Beijing 100875, China}
\email{czy@bnu.edu.cn}

\author{Yin Tian}
\address{School of Mathematical Sciences, Beijing Normal University; 
Laboratory of Mathematics and Complex Systems, Ministry of Education, Beijing 100875, China}
\email{yintian@bnu.edu.cn}

\author{Tianyu Yuan}
\address{School of Mathematical Sciences, Eastern Institute of Technology, Ningbo, Zhejiang, 315200, China}
\email{tyyuan@eitech.edu.cn}

\date{\today}

\keywords{Plane curves, strangeness, Arnold invariant, skein relation}

\begin{document}

\maketitle

\begin{abstract}
We construct a polynomial invariant of planar curves via skein relations, which generalizes Arnold's numerical invariant $\operatorname{St}$. 
\end{abstract}

\section{Introduction}\label{section1}
Consider the space of all $C^{\infty}$-immersions of $S^1$ into $\mathbb{R}^2$, Whitney's classical result tells us that two generic plane curves are regular homotopic if and only if they share the same rotation number \cite{Whitney1937}. Here we say a plane curve is \emph{generic} if all the self-intersections are a finite number of transverse double points. It follows immediately that each component of this space consists of immersions with the same rotation number. Inspired by Vassiliev's finite-type invariants of knots, Arnold breathed new life to the study of plane curves in 1990s by considering the complement of generic plane curves \cite{Arnold1994}. The complement consists of three types of discriminant hypersurfaces: the direct self-tangency perestroika (when the two velocity vectors have the same direction), the inverse self-tangency perestroika (when the two velocity vectors have opposite directions) and the triple point crossing. See Figure \ref{fig: arnold}.

\begin{figure}[ht]
    \centering
    \includegraphics[width=14cm]{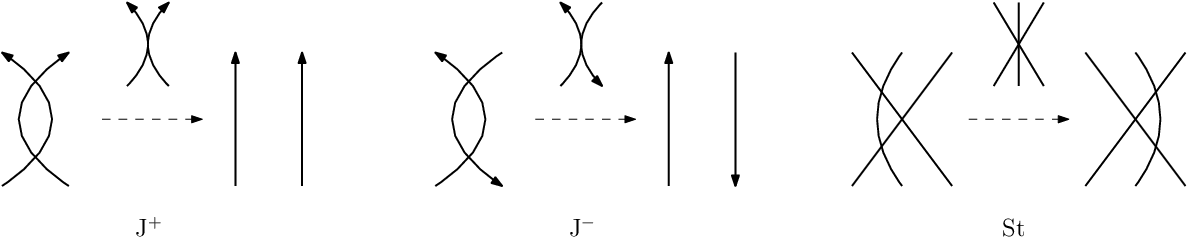}
    \caption{}   
    \label{fig: arnold}
\end{figure}

Arnold introduced three numerical invariants $\operatorname{J}^+, \operatorname{J}^-$ and strangeness $\operatorname{St}$ by assigning values to some standard curves and their jumps when a plane curve passing through a hypersurface. More precisely,
\begin{itemize}
\item $\operatorname{J}^+$ is increased by 2 when it meets a positive direct self-tangency perestroika and is preserved under the inverse self-tangency perestroika and the triple point crossing;
\item $\operatorname{J}^-$ is decreased by 2 when it meets a positive inverse self-tangency perestroika and is preserved under the direct self-tangency perestroika and the triple point crossing;
\item $\operatorname{St}$ is increased by 1 when it meets a positive triple point crossing and is preserved under the two self-tangency perestroika.
\end{itemize}
The reader is referred to \cite{Arnold1994} for more details about the coorientation of the discriminant. 

It is not a simple task to calculate Arnold's three invariants directly using the original definitions, especially for plane curves with many self-intersections. Explicit formulas for $\operatorname{J}^+$ and $\operatorname{J}^-$ in terms of topological data of plane curves were introduced by Viro in \cite{Viro1996}. For $\operatorname{St}$, in \cite{Shumakovitch1995} Shumakovitch provided three explicit formulas for it. A unifying approach to these three invariants was found by Polyak via Gauss diagrams \cite{Polyak1998}. The quantization of $\operatorname{J}^+, \operatorname{J}^-$ and $\operatorname{St}$ were realized in \cite{LP2013}, \cite{Viro1996} and \cite{Ito2023}, respectively. Besides, Arnold's invariant $\operatorname{J}^+$ has also been found to have deep connections with other research areas, such as contact geometry \cite{Ng2006} and restricted three-body problem \cite{CFK2017}.

Current research is predominantly concentrated on $\operatorname{J}^+, \operatorname{J}^-$, with comparatively few investigations into the $\operatorname{St}$ invariant. In this paper, we introduce a polynomial invariant for multi-component plane curves. Some part of it can be regarded as a generalization of the $\operatorname{St}$ invariant. 
The main idea behind this polynomial invariant is to apply an analogue of various skein relations in knot theory. 
Our starting point is the skein relation R3 corresponding to the triple point crossing; see the figure below. It is natural to expect that the resulting polynomial is related to Arnold's invariant $\op{St}$. 
This skein relation is motivated by Floer homology in symplectic orbifolds $T^*(\mathbb{C}/\mathbb{Z}_3)$ \cite{HondaKrutowskiTianYuan2026}. 
We consider {\em bulk deformations} in Floer homology for a cotangent fiber. The resulting skein relation has a diagrammatic interpretation as the skein relation R3.  

Let $R=\mathbb{Z}[t][[h]]$ be the coefficient ring, we construct an invariant 
\begin{center}
$A: \{\mbox{isotopy classes of generic oriented plane curves}\} \to R$,
\end{center}
which satisfies the following skein relations for the plane curves.

\begin{enumerate}[leftmargin=2.5em,labelsep=0.8em]
  \item[R1]
  \makebox[\linewidth][c]{%
  \(
  \includegraphics[height=1.5cm,valign=c]{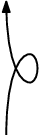}
  =
  t^{-1}
  \includegraphics[height=1.5cm,valign=c]{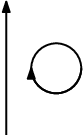},
  \qquad
  \includegraphics[height=1.5cm,valign=c]{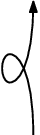}
  =
  t^{-1}
  \includegraphics[height=1.5cm,valign=c]{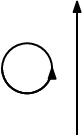}
  \)
  }

  ~\\
  \item[R2]
  \makebox[\linewidth][c]{%
  \(
  \includegraphics[height=1.5cm,valign=c]{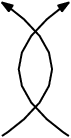}
  =
  \includegraphics[height=1.5cm,valign=c]{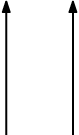}
  +
  th\,
  \includegraphics[height=1.5cm,valign=c]{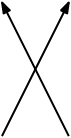}
  \)
  }

  ~\\
  \item[R2']
  \makebox[\linewidth][c]{%
  \(
  \includegraphics[height=1.5cm,valign=c]{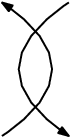}
  =
  \includegraphics[height=1.5cm,valign=c]{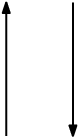},
  \qquad
  \includegraphics[height=1.5cm,valign=c]{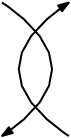}
  =
  \includegraphics[height=1.5cm,valign=c]{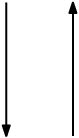}
  \)
  }

  ~\\
  \item[R3]
  \makebox[\linewidth][c]{%
  \(
  \includegraphics[height=1.5cm,valign=c]{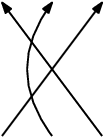}
  -
  \includegraphics[height=1.5cm,valign=c]{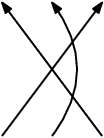}
  =
  h
  \left(
  \includegraphics[height=1.5cm,valign=c]{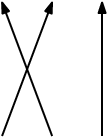}
  -
  \includegraphics[height=1.5cm,valign=c]{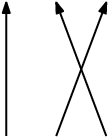}
  \right)
  \)
  }

  ~\\
  \item[R3']
  \makebox[\linewidth][c]{%
  \(
  \includegraphics[height=1.5cm,valign=c]{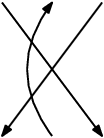}
  -
  \includegraphics[height=1.5cm,valign=c]{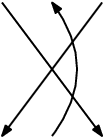}
  =
  t^{-1}h
  \left(
  \includegraphics[height=1.5cm,valign=c]{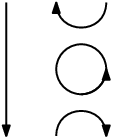}
  -
  \includegraphics[height=1.5cm,valign=c]{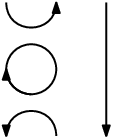}
  \right)
  \)
  }

  ~\\
  \item[R0]
  \makebox[\linewidth][c]{%
  \(
  \includegraphics[height=1.5cm,valign=c]{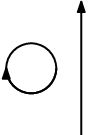}
  =
  (1+h)
  \includegraphics[height=1.5cm,valign=c]{figures/skein/unc.eps},
  \qquad
  \includegraphics[height=1.5cm,valign=c]{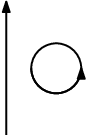}
  =
  (1+h)
  \includegraphics[height=1.5cm,valign=c]{figures/skein/cu.eps}
  \)
  }

  ~\\
  \item[I]
  \makebox[\linewidth][c]{%
  \(
  \prescript{*}{}{
  \includegraphics[height=0.7cm,valign=c]{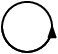}
  }
  =
  {
  \includegraphics[height=0.7cm,valign=c]{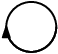}
  }^{*}
  =
  t
  \)
  }

  ~\\
  \item[D]
  \makebox[\linewidth][c]{%
  \(
  A(L_1\sqcup L_2)=
  A(L_1)A(L_2)
  \)
  }
\end{enumerate}

Here we have some remarks:
\begin{enumerate}
  \item The relations R2 and R2' correspond to the direct self-tangency perestroika and inverse self-tangency perestroika respectively; 
  \item The relations R3 and R3' correspond to the two cases of triple point crossing perestroika;
  \item The relation R0 shows the effect of moving a contractible component across an arc; 
  \item The initial-value relation I says that a contractible component near the base point at infinity is evaluated as $t$;
  \item The relations R0 and I together say that a trivial component is evaluated as $t(1+h)^k$, where $k$ is the \emph{winding number} of the plane curves about the exterior region adjacent to the trivial component;
  \item The relation D says that the invariant $A$ is multiplicative with respect to the disjoint union.
\end{enumerate}

Although $t^{-1}$ appears in relations $R1$ and $R3'$, it would be cancelled out by the initial value $t$ appeared in relation I. For $(1+h)^{-1}$ that may occur in relation R0, we expand $(1+h)^{-1}$ into a power series
\begin{center}
$(1+h)^{-1}=1-h+h^2-h^3+\cdots$.
\end{center}
Thus, the resulting invariant $A(L)$ takes value in $R=\mathbb{Z}[t][[h]]$ for any plane curves $L$. As a notation convention, throughout this paper we use $L$ to denote multi-component plane curves, and use $K$ when the component number is one.

It is not hard to observe that the eight relations above allow us to calculate $A(L)$ for any plane curve $L$. However, a natural question is whether the invariant $A$ is well-defined. In other words, if we choose two distinct computational paths for $A(L)$, do we always obtain the same value? The following result provides an affirmative answer to this question.

\begin{theorem}\label{theorem1.2}
The invariant $A(L)$ is well-defined. 
\end{theorem}

Let $L$ be a plane curve with $n(L)$ components, assume $A(L)=\sum\limits_{k\ge 0}a_k(L)h^k$, where $a_k(L) \in \mathbb{Z}[t]$. It turns out that $a_1(L)$ can recover Arnold's invariants $\operatorname{J}^+$ and $\operatorname{St}$ for one-component plane curves, and Viro's invariant $\op{J}^{+}_{\op{Vir}}$ as a generalization of $\op{J}^{+}$ for multi-component curves \cite{Viro1996}. More precisely, we have the following results.

\begin{theorem}\label{thm A}
Let $L$ be a plane curve with $n(L)$ components, then we have
\begin{enumerate}
    \item the $0$-th coefficient $a_0(L)=t^{n(L)}$;
    \item the first coefficient $a_1(L)|_{t=1}=\frac{1}{2}\op{J}^{+}_{\op{Vir}}(L)$;
    \item for a one-component plane curve $K$, the first coefficient has the form
    \begin{center}
    $a_1(K)=-\op{St}(K)t+(\frac{\op{J}^+(K)}{2}+\operatorname{St}(K))t^3\in\mathbb{Z}[t]$.
    \end{center}
    \item the polynomial $t^{-1}A(L)$ is multiplicative with respect to the connected sum. Equivalently, we have $tA(L_1\# L_2)=A(L_1)A(L_2)$.
\end{enumerate}
\end{theorem}

If we write $A(L)$ as a power series in the variable $t$, let us use $B(L)$ denote the coefficient of $t$. In other words, $B(L)=\frac{A(L)}{t}|_{t=0},$ here we use the fact that $t|A(L)$, which can be concluded from the relations above by taking $t=0$. For the polynomial $B(L)\in\mathbb{Z}[h]$, we have the following results.

\begin{theorem}\label{thm B}
Let $L$ be a plane curve with $n(L)$ components, then we have
  \begin{enumerate}
    \item the invariant $B(L)=\sum_{k\ge 0}b_k(L)h^k$ is of $\operatorname{St}$-type, i.e. it is invariant under direct self-tangency perestroika and inverse self-tangency perestroika; 
    \item the invariant $B(L)$ is multiplicative with respect to the connected sum;
    \item for a one-component curve $K$, $b_0(K)=1, b_1(K)=-\operatorname{St}$. Moreover, the $k$-th coefficient $b_k(K)$ is an $\op{St}$-type invariant of finite order $k$.
    \item the invariant $B(L)=0$ if $L$ is splittable, i.e. it is a disjoint union of two plane curves;
    \item the polynomial $B(L)$ is divisible by $h^{n(L)-1}$. Moreover, this estimate of $h$-order is sharp when $n(L)$ is odd.
  \end{enumerate}
\end{theorem}

We sketch the main idea of the proof of Theorem \ref{theorem1.2} here, instead of proving it directly, we provide another interpretation of $A(L)$. For a given plane curve with $c$ self-intersections, we first resolve each crossing into either an oriented smoothing or a positive crossing, which yields $2^c$ links in $\mathbb{R}^3$, then we use a variant of HOMFLY-PT polynomial to evaluate the resulting links. Finally, it is sufficient to check that the resulting polynomial satisfies the skein relations. We call this presentation of $A(L)$ the {\em state sum} formula. Note that Chmutov, Goryunov and Murakami used the HOMFLY-PT polynomial to study Legendrian lifts of immersed plane curves in \cite{CGM2000}.

This state sum formula induces a similar state sum formula for the polynomial $B(L)$. Roughly speaking, $B(L)$ is the average of the Conway polynomials of the $2^c$ liftings of $L$ with an extra normalization factor, see Proposition \ref{prop B sum}. Polyak considered a similar average of Vassiliev knot invariants in \cite[Section 6]{Polyak1998}. We focus on the special case of the Conway polynomial, and explore the skein relations of the corresponding invariant.

\begin{remark}
Wang considered a relation similar to R3 in the study of spin Hecke algebras \cite{Wang2007}. 
\end{remark}

We finally point out some further directions:
\begin{enumerate}
\item It is natural to consider invariants of plane curves as state sums of other link invariants, and to see if there are other interesting skein relations. 
\item The polynomial invariants admit state sum expressions in terms of link polynomials. We can further expand the link polynomial into its state sum formula. As a result, we should obtain a state sum in terms of chord diagrams for our invariant of plane curves.
\item Various link polynomials lead to skein modules of 3-manifolds. We can definitely consider a similar module for 2-manifolds. Is such a construction also related to quantization? 
\item It is expected that the skein relations in this paper have a curve counting interpretation via Floer theory of symplectic orbifolds. 
\end{enumerate}

\noindent{\em Organization:} In Section \ref{section2}, we give the state sum formula of $A(L)$. Based on it, we give a proof of Theorem \ref{theorem1.2}. Section \ref{section3} contains all the skein relations of the polynomial invariant $B(L)$. The state sum of it is also discussed. In Section \ref{section4}, we discuss the properties of our invariants and their relations to Arnold's invariants. Theorem \ref{thm A} and Theorem \ref{thm B} are also proved.

\vspace{.2cm}
\noindent {\bf Acknowledgements:} Y.T. and T.Y. thank Ko Honda for valuable discussions. The authors used generative AI to assist with the computation in deriving the state sum formula for $A(L)$, and also used it as a discussion aid in developing the proof of Theorem \ref{thm B} (5).
Z.C. is partially supported by NSFC Grant No. 12371065, Y.T. is partially supported by NSFC Grant No. 12471064, T.Y. is partially supported by NSFC Grant No. 12501080.

\section{The state sum formula of $A(L)$}\label{section2}
We give the construction of the invariant $A(L)$ via resolution and the HOMFLY-PT polynomial in this section. Consider the following auxiliary coefficient ring 
\begin{gather} \label{eq ring}
\tilde{R}=\mathbb{Z}[t,h,(1+h)^{-1},q^{\pm 1},c,c']/(q^2=1+h, c+c'=-th, cc'=h).
\end{gather}
The construction of $A(L)$ consists of three steps:
\begin{enumerate}
\item resolve the crossings to obtain a collection of links in $\mathbb{R}^3$;
\item evaluate each link by a variant of HOMFLY-PT polynomial;
\item add a normalization.
\end{enumerate}

For each self-intersection ``$X$", consider the following two resolutions: the positive crossing ``$X_+$" and the oriented smoothing ``$I$". A local skein relation is given by
\begin{equation}
    \label{eq-skein}
    \big[\includegraphics[height=1.5cm,valign=c]{figures/skein/uux.eps}\big]\,=\,q\big[\includegraphics[height=1.5cm,valign=c]{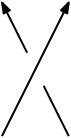}\big]\,-\,c\,\big[\,\includegraphics[height=1.5cm,valign=c]{figures/skein/id.eps}\,].
\end{equation}
Let $L$ be a plane curve with $c(L)$ self-intersections. After resolving all the self-intersections we obtain $2^{c(L)}$ links $L_{\epsilon}$ in $\mathbb{R}^3$. For each link $L_{\epsilon}$, we use $[L_{\epsilon}]$ to denote its framed HOMFLY-PT polynomial evaluated in $\mathbb{Z}[a^{\pm 1},z^{\pm 1}]$, which is defined by the following skein relations: 
\begin{equation} \label{eq-homfly}
    \includegraphics[height=1.5cm,valign=c]{figures/skein/poscross.eps}\,-\,\includegraphics[height=1.5cm,valign=c]{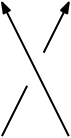}\,=\,z\,\includegraphics[height=1.5cm,valign=c]{figures/skein/id.eps},\qquad \includegraphics[height=1.5cm,valign=c]{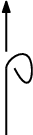}\,=\,a\,\includegraphics[height=1.5cm,valign=c]{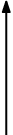},\qquad \includegraphics[height=1.5cm,valign=c]{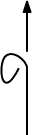}\,=\,a^{-1}\,\includegraphics[height=1.5cm,valign=c]{figures/skein/id1.eps},
\end{equation}
and the evaluation on the unknot $O$ is $[O]=\frac{a-a^{-1}}{z}$. Here the variables $z$ and $a$ satisfy the following equations:
\begin{gather} \label{eq act}
a=1+ct, \quad a^{-1}=1+c't, \quad qz=c-c'.
\end{gather}
It follows that $[O]=\frac{a-a^{-1}}{z}=(c-c')t=qt$ and $$[L]=\sum_{\epsilon} q^{x(\epsilon)}(-c)^{i(\epsilon)}[L_{\epsilon}],$$
where $x(\epsilon)$ and $i(\epsilon)$ denote the numbers of $X_+$-resolutions and $I$-resolutions of the state $\epsilon$, respectively.   

Finally, we add a normalization factor as follows. Let $\tilde{L}$ be a collection of disjoint circles obtained from $L$ by applying smoothing at every self-intersection along the orientation. Denote the components of $\tilde{L}$ by $C_1(L), \dots, C_{r(L)}(L)$. Note that $r(L)$ is nothing but the number of Seifert circles for any $L_{\epsilon}$. Set $w_i=1$ if $C_i$ is counterclockwise oriented and $w_i=-1$ if $C_i$ is clockwise oriented. Define $$\eta(L)=-\sum_{C_j \subset \op{int}(C_i)}w_iw_j,$$ 
here the sum takes over all the pairs $\{C_i, C_j\}$ such that one contained within another, and we set $$s(L)=2\eta(L)-r(L).$$

\begin{definition}\label{definition2.1}
For a generic plane curve $L$, we define $A(L)=q^{s(L)}[L] \in \tilde{R}$.
\end{definition}

We call this the {\em state sum} formula of $A(L)$. Next we claim that this definition actually coincides with that one we give in Section \ref{section1}. In order to prove this, it suffices to prove that the state sum formula satisfies all the skein relations in terms of $t$ and $h$. After that, we will find that $A(L)$ actually takes values in $R=\mathbb{Z}[t, h]$. Before proving Theorem \ref{theorem1.2}, we give the following simple but useful lemma.

\begin{lemma}\label{lemma2.2}
The local relation $[X]=q[X_-]-c'[I]$ holds, where $X_-$ is obtained from $X$ by replacing the self-intersection with a negative crossing.
\end{lemma}
\begin{proof}
According to the HOMFLY-PT relation $[X_+]-[X_-]=z[I]$, we have
\begin{center}
$[X]=q[X_+]-c[I]=q([X_-]+z[I])-c[I]=q[X_-]+(qz-c)[I]=q[X_-]-c'[I]$.
\end{center}
\end{proof}

Now we give the proof of Theorem \ref{theorem1.2}.
\begin{proof}
\newcommand{\proofpic}[2][1.5cm]{\includegraphics[height=#1,valign=c]{#2}}
We show that $A(L)$ defined by the state sum formula satisfies the skein relations R1, R2, R2', R3, R3', R0, I and D case by case. First we gather the following relations
\begin{align} \label{eq var}
q^2=1+h, c+c'=-th, cc'=h, qz=c-c', a=1+ct, a^{-1}=1+c't,
\end{align}
together with $[O]=qt$. All diagrammatic calculations below are about the HOMFLY-PT state sum.

\noindent(R1)
We first compute its HOMFLY-PT state sum:
\[
\begin{aligned}
\proofpic{figures/skein/cr.eps}
&=q\,\proofpic{figures/skein/curlr.eps}
  -c\,\proofpic{figures/skein/unc.eps}
 =(qa-cqt)\,\proofpic{figures/skein/id1.eps}  \\
&=q\,\proofpic{figures/skein/id1.eps}
 =t^{-1}\,\proofpic{figures/skein/unc.eps}
\end{aligned}
\]
The other case of R1 can be verified similarly. Since the normalization exponent is unchanged, this proves that $A(L)$ satisfies the skein relation R1.

\noindent(R2)
We first compute its HOMFLY-PT state sum and use Lemma \ref{lemma2.2}:
\[
\proofpic{figures/skein/uuxx.eps}
={}q^2\,\proofpic{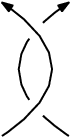}
 -qc\,\proofpic{figures/skein/negcross.eps}
 -qc'\,\proofpic{figures/skein/poscross.eps}
 +cc'\,\proofpic{figures/skein/id.eps}
\]
The HOMFLY-PT skein relation implies that
\[
\begin{aligned}
\proofpic{figures/skein/uuxx.eps}
={}&(q^2+qcz+cc')\,\proofpic{figures/skein/id.eps}
 -q(c+c')\,\proofpic{figures/skein/poscross.eps}\\
={}&(q^2+c^2)\,\proofpic{figures/skein/id.eps}
 -q(c+c')\,\proofpic{figures/skein/poscross.eps}\\
={}&(1-thc)\,\proofpic{figures/skein/id.eps}
 +qth\,\proofpic{figures/skein/poscross.eps}\\
={}&\proofpic{figures/skein/id.eps}
 +th\,\proofpic{figures/skein/uux.eps}
\end{aligned}
\]
Here we are using $q^2+c^2=1-thc, c+c'=-th$ and Lemma \ref{lemma2.2}.  The normalization exponent is unchanged, this proves the case of R2.

\noindent(R2')
Let
\[
C=\proofpic{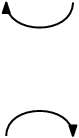},
\qquad
\proofpic{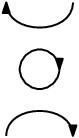}=qtC.
\]
The HOMFLY-PT state sum gives
\[
\begin{aligned}
\proofpic{figures/skein/udxx.eps}
={}&q^2\,\proofpic{figures/skein/ud.eps}
 -qca^{-1}C-qc'aC+cc'qtC\\
={}&q^2\,\proofpic{figures/skein/ud.eps}
 -q(ca^{-1}+c'a-cc't)C
 =q^2\,\proofpic{figures/skein/ud.eps},
\end{aligned}
\]
because
$ca^{-1}+c'a-cc't=c(1+c't)+c'(1+ct)-cc't=c+c'+cc't=-th+ht=0$.
A direct comparison of the complete smoothings gives
\[
s(\proofpic{figures/skein/udxx.eps})
=s(\proofpic{figures/skein/ud.eps})-2.
\]
The factor $q^2$ is canceled by the normalization factor. The other case of R2' can be proved in the same way.  

\noindent(R3)
Set
\[
P_1=\proofpic{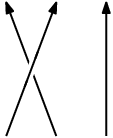},\quad
P_2=\proofpic{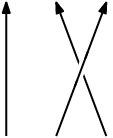},\quad
U_1=\proofpic{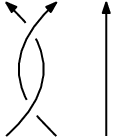},\quad
U_2=\proofpic{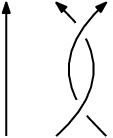}.
\]
The HOMFLY-PT relation gives $U_1-U_2=z(P_1-P_2)$. We also have
\[
\proofpic{figures/skein/uuuxi.eps}
-\proofpic{figures/skein/uuuix.eps}
=q(P_1-P_2).
\]
Expanding the three crossings and canceling the isotopic state pairs, we
obtain
\[
\begin{aligned}
\proofpic{figures/skein/uuul.eps}
 -\proofpic{figures/skein/uuur.eps}&=-q^2c(U_1-U_2)+qc^2(P_1-P_2)\\
&=(-q^2cz+qc^2)(P_1-P_2)\\
&=qcc'(P_1-P_2)\\
&=qh(P_1-P_2)\\
&=h\left(\proofpic{figures/skein/uuuxi.eps}-\proofpic{figures/skein/uuuix.eps}\right).
\end{aligned}
\]
These four plane diagrams have the same normalization factors.  This proves the case of R3.

\noindent(R3')
Remove the detached circle from the two correction diagrams and put
\[
\Delta=
\proofpic{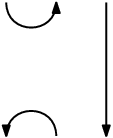}
-\proofpic{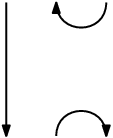}.
\]
Since the circle in HOMFLY-PT has value $qt$,
\[
\proofpic{figures/skein/dudc.eps}
-\proofpic{figures/skein/dudcl.eps}=-qt\Delta.
\]
Expand the eight states on each side, and let $E_k$ be the contribution
to their difference from the states having exactly $k$ $I$-resolutions so that
\newcommand{\rpp}[1]{\proofpic[1.4cm]{figures/proof/r3p-#1.eps}}%
\[
\proofpic{figures/skein/dudl.eps}
-\proofpic{figures/skein/dudr.eps}
=E_0+E_1+E_2+E_3.
\]
These four resolutions count are evaluated as follows.  The $E_0$ pair is
reduced by two HOMFLY-PT relations and a positive framed curl.  Each of
the three $E_1$ pairs uses one HOMFLY-PT relation and one positive curl;
the $E_2$ pairs are positive curls on the two core diagrams; and the
$E_3$ pair is the detached circle on the two core diagrams.  Consequently,
\[
\begin{aligned}
E_0
&=q^3\bigl(\rpp{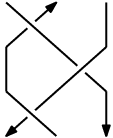}-\rpp{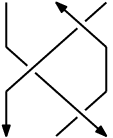}\bigr)
 =q^3z^2a\Delta
 =qa(c-c')^2\Delta,\\
E_1&=-q^2c\Bigl(
 \bigl(\rpp{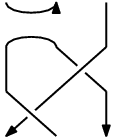}-\rpp{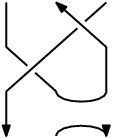}\bigr)
 +\bigl(\rpp{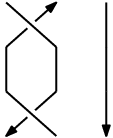}-\rpp{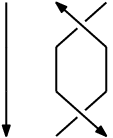}\bigr)\\
&+\bigl(\rpp{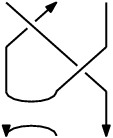}-\rpp{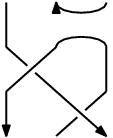}\bigr)\Bigr)
 =-3q^2cza\Delta
 =-3qca(c-c')\Delta,\\
E_2&=qc^2\Bigl(
 \bigl(\rpp{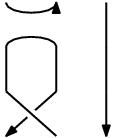}-\rpp{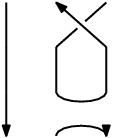}\bigr)
 +\bigl(\rpp{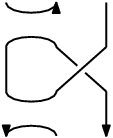}-\rpp{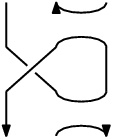}\bigr)\\
&+\bigl(\rpp{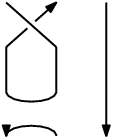}-\rpp{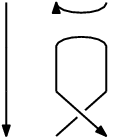}\bigr)\Bigr)
 =3qc^2a\Delta,\\
E_3
&=-c^3\left(
 \proofpic[1.4cm]{figures/skein/dudcl.eps}
 -\proofpic[1.4cm]{figures/skein/dudc.eps}\right)
 =-c^3qt\Delta.
\end{aligned}
\]
Collecting the four cases and using $qz=c-c'$ now gives
\[
\begin{aligned}
E_0+E_1+E_2+E_3
&=q\big(a(c-c')^2-3ac(c-c')+3ac^2-c^3t\big)\Delta\\
&=q\big(a(c^2+cc'+c'^2)-c^3t\big)\Delta\\
&=-qcc'\Delta\\
&=-qh\Delta.
\end{aligned}
\]
For the last equality, observe that
\[
a(c^2+cc'+c'^2)-c^3t+cc'
=(c+c')\big((c+c')+tcc'\big)=0.
\]
Consequently,
\[
\proofpic{figures/skein/dudl.eps}
-\proofpic{figures/skein/dudr.eps}
=t^{-1}h(
\proofpic{figures/skein/dudc.eps}
-\proofpic{figures/skein/dudcl.eps}).
\]
All these four diagrams have the same normalization factors, so the proof of R3' case is finished.

\noindent(R0)
The HOMFLY-PT evaluations on the two sides are equal, while the exponent $s$ of the left-hand diagram is two larger than
that of the right-hand diagram.  Thus the factor is
$q^2=1+h$.

\noindent(I)
We evaluate on the trivial circle $O$, $$A(O)=q^{-1}[O]=q^{-1}\frac{a-a^{-1}}{z}=\frac{ct-c't}{qz}=\frac{qzt}{qz}=t.$$

\noindent(D)
Both the HOMFLY-PT state sums and the normalization factors are multiplicative under the operation of disjoint union. This proves that $A(L_1\cup L_2)=A(L_1)A(L_2)$.
\end{proof}

\section{The polynomial invariant $B(L)$}\label{section3}
For a plane curve $L$, recall that $B(L)\in\mathbb{Z}[h]$ is defined as $B(L)=\frac{A(L)}{t}|_{t=0}$. In other words, $A(L)=B(L)t+\text{higher-order terms in }t$. We first list all the skein relations of the polynomial invariants $B(L)$, which directly follow from those of $A$.

\begin{enumerate}[leftmargin=2.5em,labelsep=0.8em]

  \item[R1]
  \makebox[\linewidth][c]{%
  \(
  \includegraphics[height=1.5cm,valign=c]{figures/skein/cr.eps}^k
  =
  (1+h)^k\,
  \includegraphics[height=1.5cm,valign=c]{figures/skein/id1.eps},
  \qquad
  \includegraphics[height=1.5cm,valign=c]{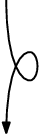}^k
  =
  (1+h)^{-k}\,
  \includegraphics[height=1.5cm,valign=c]{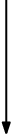}
  \)
  }

  ~\\
  \item[R2]
  \makebox[\linewidth][c]{%
  \(
  \includegraphics[height=1.5cm,valign=c]{figures/skein/uuxx.eps}
  \,=\,
  \includegraphics[height=1.5cm,valign=c]{figures/skein/id.eps}
  \)
  }

  ~\\
  \item[R2']
  \makebox[\linewidth][c]{%
  \(
  \includegraphics[height=1.5cm,valign=c]{figures/skein/udxx.eps}
  =
  \includegraphics[height=1.5cm,valign=c]{figures/skein/ud.eps},
  \qquad
  \includegraphics[height=1.5cm,valign=c]{figures/skein/duxx.eps}
  =
  \includegraphics[height=1.5cm,valign=c]{figures/skein/du.eps}
  \)
  }

  ~\\
  \item[R3]
  \makebox[\linewidth][c]{%
  \(
  \includegraphics[height=1.5cm,valign=c]{figures/skein/uuul.eps}
  -
  \includegraphics[height=1.5cm,valign=c]{figures/skein/uuur.eps}
  \,=\,
  h
  \left(
  \includegraphics[height=1.5cm,valign=c]{figures/skein/uuuxi.eps}
  -
  \includegraphics[height=1.5cm,valign=c]{figures/skein/uuuix.eps}
  \right)
  \)
  }

  ~\\
  \item[R3']
  \makebox[\linewidth][c]{%
  \(
  \includegraphics[height=1.5cm,valign=c]{figures/skein/dudl.eps}
  -
  \includegraphics[height=1.5cm,valign=c]{figures/skein/dudr.eps}
  \,=\,
  h
  \left(
  \includegraphics[height=1.5cm,valign=c]{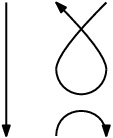}
  -
  \includegraphics[height=1.5cm,valign=c]{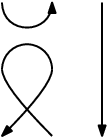}
  \right)
  \,=\,
  h
  \left(
  \includegraphics[height=1.5cm,valign=c]{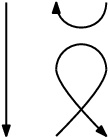}
  -
  \includegraphics[height=1.5cm,valign=c]{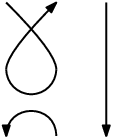}
  \right)
  \)
  }

  ~\\
  \item[I]
  \makebox[\linewidth][c]{%
  \(
  {
  \includegraphics[height=0.7cm,valign=c]{figures/skein/sc.eps}
  }^{*}
  \,=\,
  {
  \includegraphics[height=0.7cm,valign=c]{figures/skein/ncs.eps}
  }^{*}
  \,=\,
  1
  \)
  }

  ~\\
  \item[D]
  \makebox[\linewidth][c]{%
  B\((L_1\sqcup L_2)=0\)}
\end{enumerate}
In Relation R1, the integer $k$ denotes the winding number of the plane curve about the region marked with the letter $k$. It changes by one when crossing an arc, see Figure \ref{fig: id1_label}. The relation R3' is derived from the relations R1 and R3' for the invariant $A(L)$.

\begin{figure}[ht]
    \centering
    \includegraphics[width=2cm]{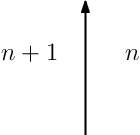}
    \caption{}   
    \label{fig: id1_label}
\end{figure}

It follows from R2 and R2' that the polynomial $B(L)$ is invariant under the two self-tangency perestroika, so that it is an invariant of $\op{St}$ type.  Next we give the construction of the polynomial invariant $B(L)$ via resolutions. The state sum formula of $A(L)$ uses the HOMFLY-PT polynomial. Next we show that a similar state sum for $B(L)$ can be obtained via the Conway polynomial. 

Recall that $B(L)=\frac{A(L)}{t}|_{t=0}$. The relations between variables in Equation \ref{eq var} are simplified by setting $t=0$:
\begin{align} \label{eq var2}
q^2=1+h,\qquad c+c'=0,\qquad c^2=-h, \qquad qz=2c,\qquad a=1.
\end{align}
Now Lemma \ref{lemma2.2} implies that 
$$[X]=\frac{1}{2}(q[X_+]-c[I]+q[X_-]-c'[I])=\frac{q}{2}([X_+]+[X_-]),$$
since $c+c'=0$.
The $a=1$ specialization of the HOMFLY-PT polynomial skein relations (\ref{eq-homfly}) coincide with that of the Conway polynomial $\nabla$.
The variable $z$ in $\nabla$ is related to $h$ via 
\begin{gather} \label{eq z}
z^2=\frac{4c^2}{q^2}=\frac{-4h}{1+h}.
\end{gather}
The following expression of $B(L)$ is mainly derived from the state sum formula of $A(L)$.

\begin{proposition}\label{prop B sum}
The invariant $B(L)=q^{s(L)+c(L)+1}\overline{\nabla}(L)$, where $$\overline{\nabla}=2^{-c(L)}\sum_{\epsilon \in \{+, -\}^{c(L)}}\nabla(L_{\epsilon})$$ 
is the average of the Conway polynomials of the $2^{c(L)}$ lifts of $L$. Here $c(L)$ denotes the number of self-intersections of $L$.
\end{proposition}
\begin{proof}
According to Definition \ref{definition2.1}, we have
\begin{center}
$A(L)=q^{s(L)}[L]=q^{s(L)}\sum\limits_{\epsilon}q^{x(\epsilon)}(-c)^{i(\epsilon)}[L_{\epsilon}]$,
\end{center}
here $[L_{\epsilon}]$ denotes the HOMFLY-PY polynomial of $L_{\epsilon}$, which is defined by the skein relations (\ref{eq-homfly}). For a fixed $[L_{\epsilon}]$, by using the skein relations (\ref{eq-homfly}) it can be written as the linear combination of HOMFLY-PT polynomials evaluated on some unlinks. Since $[O]=qt$, $\frac{[L_{\epsilon}]}{t}$ can be obtained by renormalizing $[O]=q$. Note that here we essentially use the fact that the (unnormalized) Conway polynomial of an unlink with more than one component vanishes. According to the discussion above, now we have 
\begin{center}
$\frac{[L]}{t}|_{t=0}=(\frac{q}{2})^{c(L)}\sum\limits_{\epsilon \in \{+, -\}^{c(L)}}\nabla'(L_{\epsilon})$.
\end{center}
Here $\nabla'$ is the unnormalized Conway polynomial which evaluated on the unknot equals $q$. It follows that
\begin{center}
$B(L)=\frac{A(L)}{t}|_{t=0}=q^{s(L)}\frac{[L]}{t}|_{t=0}=q^{s(L)}q^{c(L)}2^{-c(L)}\sum\limits_{\epsilon \in \{+, -\}^{c(L)}}q\nabla(L_{\epsilon})=q^{s(L)+c(L)+1}\overline{\nabla}(L)$.
\end{center}
Here we use the fact that $\nabla'(L_{\epsilon})=q\nabla(L_{\epsilon})$, since $\nabla(O)=1$.
\end{proof}

\begin{remark}
Arnold's invariant $\op{J}^+$ was extended by Viro from one-component plane curves to $n$-component plane curves in \cite{Viro1996}, which is denoted by $\op{J}^+_{\text{Vir}}$ and equals to the exponent $s(L)+c(L)+1$ of $q$ in Proposition \ref{prop B sum}. In particular, for a one-component plane curve $K$, we have $s(K)+c(K)+1=\op{J}^+(K)$.
\end{remark}

\section{Some applications}\label{section4}
We discuss the properties of our invariants $A(L)$ and $B(L)$ and their connections to Arnold's invariants in this section.

\subsection{The invariant $A(L)$}
The first formal property is about orientation reversal. Let $\bar{L}$ denote the curve obtained by simultaneously reversing orientations of all components of $L$.
\begin{proposition}\label{proposition4.1}
The invariant $A(L)$ is invariant under orientation reversal, i.e. $A(\bar{L})=A(L)$.
\end{proposition}
\begin{proof}
This holds simply because all the skein relations are invariant under orientation reversal.
\end{proof}

Now we give the proof of Theorem \ref{thm A}.
\begin{proof}
Let $L$ be a plane curve, assume $A(L)=\sum\limits_{k\geq0}a_k(L)h^k$, where $a_k(L)\in\mathbb{Z}[t]$.
\begin{enumerate}
\item When $h=0$, all the skein relations of $A(L)$ given in Section \ref{section1} become trivial, i.e. $a_0(L)$ is preserved under all the (flat) Reidemeister moves. So $a_0(L)=A(L)|_{h=0}=t^{n(L)}$ by the relation of disjoint union.
\item The number $a_1(L)|_{t=1}$ is unchanged under R2' by definition and unchanged under R3 and R3' since $a_0|_{t=1}=1$. Also due to the fact $a_0|_{t=1}=1$, it changes by one under R2. Comparing with the definition of $\op{J}^{+}_{\op{Vir}}(L)$ given in \cite{Viro1996}, it is not difficult to find that $a_1(L)|_{t=1}=\frac{1}{2}\op{J}^{+}_{\op{Vir}}(L)$.
\item For the case of one-component curves, we check the change of $a_1$ under R2, R3 and R3': 

\begin{equation}
    \label{eq-R2-coef}
    a_1\big(\includegraphics[height=1.5cm,valign=c]{figures/skein/uuxx.eps}\big)\,-\,a_1\big(\includegraphics[height=1.5cm,valign=c]{figures/skein/id.eps}\big)\,=\,t a_0\,\big(\,\includegraphics[height=1.5cm,valign=c]{figures/skein/uux.eps}\,\big)=t^3.
\end{equation}

\begin{align}
    \label{eq-R3-coef}
    a_1\big(\includegraphics[height=1.5cm,valign=c]{figures/skein/uuul.eps}\big)\,-\,a_1\big(\includegraphics[height=1.5cm,valign=c]{figures/skein/uuur.eps}\big)\,=\,a_0\,\big(\,\includegraphics[height=1.5cm,valign=c]{figures/skein/uuuxi.eps}\,\big)-a_0\big(\,\includegraphics[height=1.5cm,valign=c]{figures/skein/uuuix.eps}\,\big)\\
    = \begin{cases} t^3 - t \qquad\includegraphics[height=3cm,valign=c]{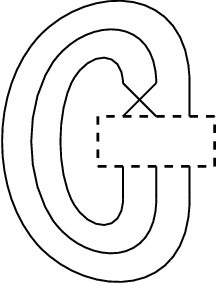}\\ \\t - t^3 \qquad\includegraphics[height=3cm,valign=c]{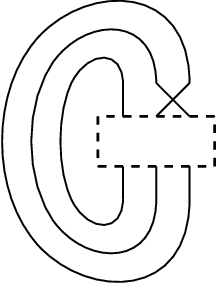}\end{cases}\nonumber
\end{align}

\begin{align}
    \label{eq-R3-1-coef}
    a_1\big(\includegraphics[height=1.5cm,valign=c]{figures/skein/dudl.eps}\big)\,-\,a_1\big(\includegraphics[height=1.5cm,valign=c]{figures/skein/dudr.eps}\big)\,=\,a_0\,\big(\,\includegraphics[height=1.5cm,valign=c]{figures/skein/dudc1.eps}\,\big)-a_0\big(\,\includegraphics[height=1.5cm,valign=c]{figures/skein/dudcl1.eps}\,\big)\\
    = \begin{cases} t - t^3 \qquad \includegraphics[height=2cm,valign=c]{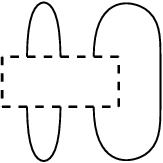}\\\\ t^3 - t \qquad\includegraphics[height=2cm,valign=c]{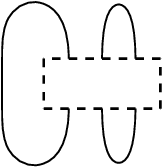}\end{cases}\nonumber
\end{align}

For R2, let $K_2, L_1, K_0$ be the plane curves with $2,1,0$ self-intersections locally. Then $K_2, K_0$ are one-component curves, and $L_1$ has two components. So $$a_1(K_2)-a_1(K_0)=ta_0(L_1)=t^3.$$

For R3 or R3', let $K, K'$ be the curves with three self-intersections, and $L, L'$ be the curves with one self-intersection. Since $K, K'$ are both one-component curves, one of $L, L'$ has one component and the other has three components depending on the connecting pattern of $K$ outside the local part. It follows that $$a_1(K)-a_1(K')=a_0(L)-a_0(L')=\pm(t-t^3).$$
Thus we have $a_1(K)=b_1(K)t+c_1(K)t^3.$ Setting $t=1$, $b_1(K)+c_1(K)=a_1(K)|_{t=1}$ is invariant under R3, hence a $\op{J}^+$-type invariant. For R2, $a_1(K_2)|_{t=1}-a_1(K_0)|_{t=1}=1$. So we have $b_1(K)+c_1(K)=\frac{1}{2}\op{J}^+(K)$. For $b_1(K)$, the change under R3 is $b_1(K)-b_1(K')=\pm 1$. Comparing with the local change of $\op{St}(K)$, we conclude that $b_1(K)=-\op{St}(K)$.
\item Recall that $A(L)=q^{s(L)}[L]$, here $s(L)=2\eta(L)-r(L)$. It is easy to see that
\begin{center}
$\eta(L_1\#L_2)=\eta(L_1)+\eta(L_2)$, $r(L_1\#L_2)=r(L_1)+r(L_2)-1$.
\end{center}
On the other hand, we have $qt[L_1\#L_2]=[L_1][L_2]$. It follows that
\begin{center}
$tA(L_1\#L_2)=tq^{s(L_1\#L_2)}[L_1\#L_2]=tqq^{s(L_1)}q^{s(L_2)}(qt)^{-1}[L_1][L_2]=A(L_1)A(L_2)$.
\end{center}
\end{enumerate}
\end{proof}

\begin{remark}\label{remark4.2}
We end this subsection with two remarks:
\begin{enumerate}
  \item For a plane curve $L$ with component number $n(L)$, the coefficient of $h^k$ has the form 
  \begin{center}
  $a_k(L)=\sum\limits_{i=0}^ka_{k, n+2i}t^{n+2i}$.
  \end{center}
  \item For a plane curve $K$ with component number one, $a_1(K)$ has the form 
  \begin{center}
  $a_1(K)=-\op{St}(K)t+\frac{1}{2}(\op{J}^+(K)+2\op{St}(K))t^3$.
  \end{center}
  It was pointed out by Arnold in \cite{Arnold1994} that all tree-like plane curves satisfy $\op{J}^++2\op{St}=0$. Here a plane curve is called \emph{tree-like} if its Gauss diagram consists of nonintersecting chords. Later, Polyak reproved this result by interpreting $\op{J}^++2\op{St}$ as a count of intersecting chord pairs in the Gauss diagram \cite[Corollary 2]{Polyak1998}. It is not difficult to observe that the coefficient of $t^3$ in $a_1(K)$ vanishes, since a tree-like curve can be transformed into a disjoint union of trivial circles just using R1. This offers a new perspective on the fact that all tree-like plane curves satisfy $\op{J}^++2\op{St}=0$. On the other hand, Arnold suggested that $\op{J}^++2\op{St}$ should be compared with functions having a zero of certain order at the origin in the same sense in which the Vassiliev invariants behave like polynomials. The coefficients of $a_k$ provide a possibility to realize Arnold's vision.
\end{enumerate}
\end{remark}

\subsection{The invariant $B(L)$}
As we mentioned before, the invariant $B(L)$ is of $\op{St}$-type, i.e. it is invariant under R2 and R2'. Before further discussing the properties of $B(L)$, we first give an example of computing the polynomial $B(L)$.
  
\begin{example} \label{ex B}
Let $L_n$ denote the closure of a braid-like plane curve $(s_1s_2)^n$ with three strands for $n\ge 0$, see Figure \ref{fig: 6-1}.
\begin{figure}
    \centering
    \includegraphics[width=13cm]{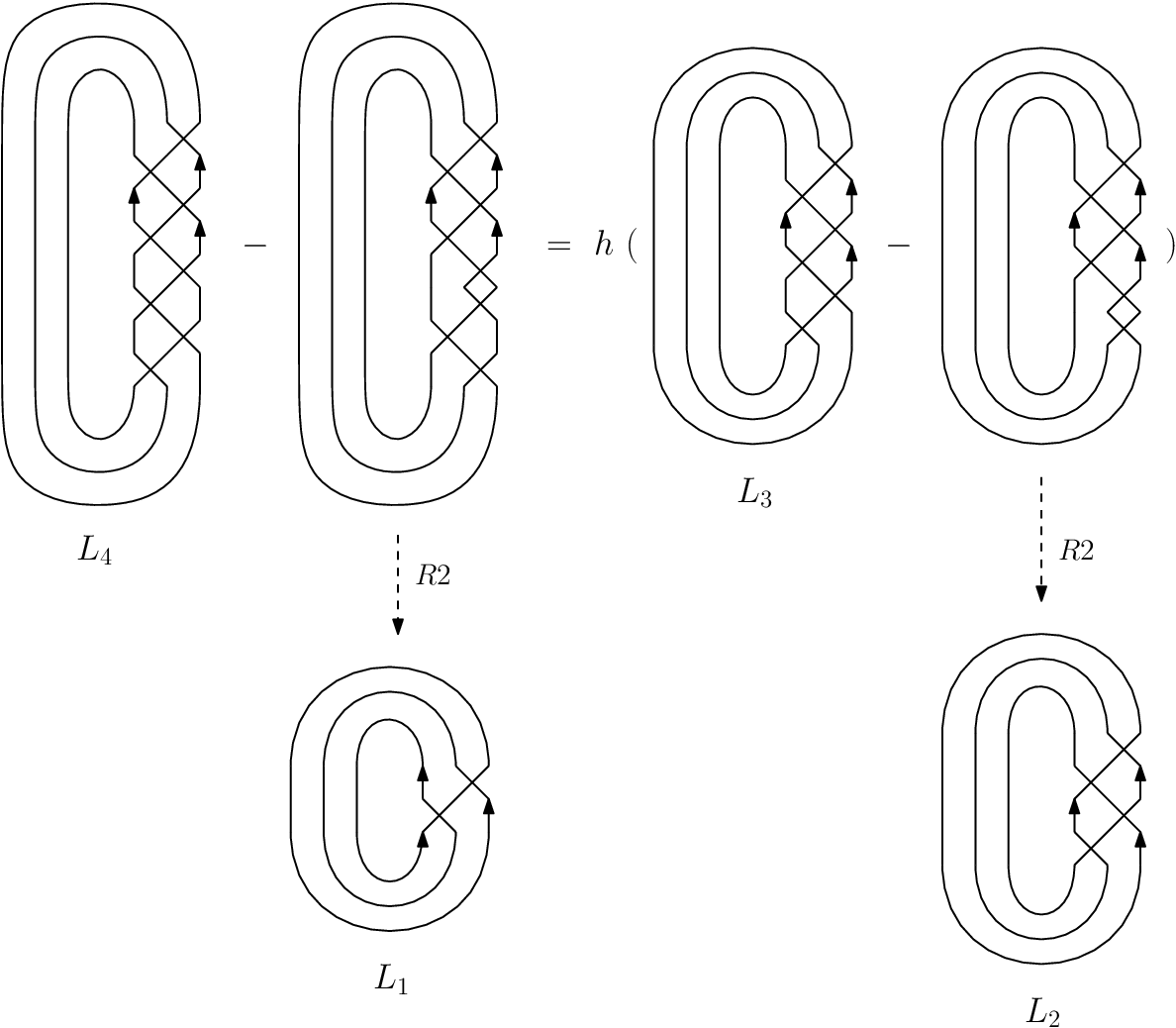}
    \caption{}
    \label{fig: 6-1}
\end{figure}

The skein relation R3 implies that 
$$B(L_n)-B(L_{n-3})=h(B(L_{n-1})-B(L_{n-2})),$$
since $B$ is invariant under R2. Then we have 
\begin{align*}
B(L_0)=0, \quad B(L_1)=(1+h)^{-3}, \quad B(L_2)=(1+h)^{-2}, \\
 B(L_3)=h^2(1+h)^{-3}, \quad B(L_4)=(1-h)^2(1+h)^{-2}.
\end{align*}
Note that $L_3$ is the shadow of the Borromean ring, and $h^2|B(L_3)$.
\end{example}

Now we turn to the proof of Theorem \ref{thm B}.
\begin{proof}
Suppose $L$ is a plane curve with $n(L)$ components, we shall now prove the five assertions of Theorem \ref{thm B} individually.
\begin{enumerate}
\item This has been verified in Section \ref{section3}.
\item According to Theorem \ref{thm A} (4), we have $A(L_1\# L_2)=t^{-1}A(L_1)A(L_2)$. Thus $B(L_1\# L_2)=B(L_1)B(L_2)$.
\item We restrict ourselves to a one-component plane curve $K$, the identities $b_0(K)=1, b_1(K)=-\op{St}$ follow directly from Theorem \ref{thm A} (1) and (3).

Recall that an $\op{St}$ type invariant $V$ is of \emph{finite order $n$} if $V$ evaluates as zero on any plane curve with $m$ singular triple points for $m>n$, and there exists a plane curve with $n$ singular tripe points such that $V$ evaluated on it is nonzero. For instance, Arnold's $\op{St}$ is of order one \cite{Arnold1994}. The skein relations R3 and R3' imply that the coefficient $b_n$ is of order at most $n$ by induction on $n$.

We show that it is of order exactly $n$ as follows. Let $K_s$ be the plane curve with one singular triple point depicted in Figure \ref{fig: h_order_1}. It interpolates between the closures $K_0$ of $s_1s_2s_1s_1$ and $K_1$ of $s_2s_1s_2s_1$. It follows from Example \ref{ex B} that $B(K_0)=B(L_1), B(K_1)=B(L_2)$, and 
$$B(K_s)=B(K_1)-B(K_0)=h(1+h)^{-3}=h-3h^2+\cdots.$$
Let $K_s^n$ denote a connected sum of $n$ copies of $K_s$, with $n$ singular triple points. The multiplicative property implies that $B(K_s^n)=B(K_s)^n=h^n(1+h)^{-3n}.$ Thus $b_n(K_s^n)=1\neq0$. Hence, $b_n$ is of order exactly $n$.

\begin{figure}[ht]
    \centering
    \includegraphics[width=3cm]{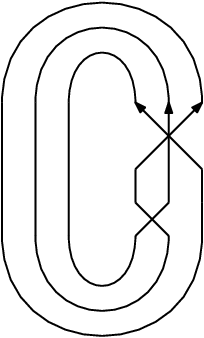}
    \caption{}   
    \label{fig: h_order_1}
\end{figure}

\item If $L$ is a disjoint union, then $t^2|A(L)$ follows from the disjoint union property D. It follows that $B(L)=0$. 
\item We consider two cases depending on the parity of $n(L)$:

If $n(L)$ is even, then $A(L)$ only contains even powers of $t$, since each skein relation preserves the parity of the exponents of $t$. See Remark \ref{remark4.2} (1). Thus we have $B(L)=0$, the result holds. 

If $n(L)$ is odd, it suffices to consider the case that $L$ is nonsplittable, otherwise $B(L)=0$. See Theorem \ref{thm B} (4). Recall the state sum formula of $B(L)$ as the average of Conway polynomials in Proposition \ref{prop B sum}. In this case, $s(L)+c(L)+1=\op{J}^+_{\text{Vir}}(L)$ is even, thus $h\nmid q^{s(L)+c(L)+1}$. We introduce variables $x, r$ such that
$$z=x-x^{-1}, \quad r=\frac{x+x^{-1}}{2},$$
where $z$ is the variable in Conway polynomial. According to the relation between the variable $z$ and $h$ (\ref{eq z}), as infinitesimals, we have $h \sim z^2 \sim r-1$. It suffices to show that $(r-1)^{n(L)-1}| \overline{\nabla}(L)$, where 
\begin{center}
$\overline{\nabla}(L)=2^{-c(L)}\sum_{\epsilon \in \{+,-\}^{c(L)}}\nabla(L_{\epsilon}).$
\end{center}

We proceed to prove this result in three steps:
\begin{enumerate}
\item express each $\nabla(L_{\epsilon})$ as a determinant via the Kauffman states; 
\item express the average $\overline{\nabla}(L)$ as a determinant of a single matrix in the variable $r$; 
\item compute the order of $r-1$ of this determinant.
\end{enumerate}

\noindent 
Step (a):  Kauffman expresses the Conway polynomial of a link diagram $L'$ as a sum over Kauffman states \cite{Kauffman1983}. This state sum is the permanent $\op{Perm}(G(L'))$ of the following matrix $G(L')$. Since $L$ is not a disjoint union, the link diagram $L'$ has $c$ crossings and $c+2$ regions. The matrix $\tilde{G}(L')$ is an $c \times (c+2)$ matrix, where rows are indexed by crossings and columns are indexed by regions. The entries are given by the local weights as in the upper line of Figure \ref{fig: h_order_2}. The matrix $G(L')$ is obtained from $\tilde{G}(L')$ by deleting two columns corresponding to two adjacent regions. 

\begin{figure}[ht]
    \centering
    \includegraphics[width=6cm]{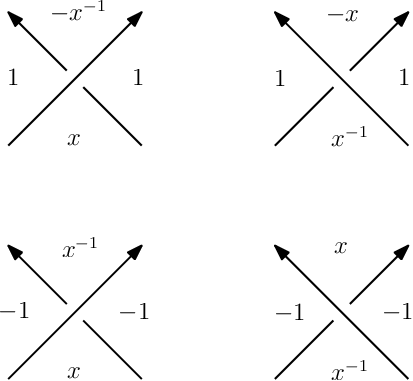}
    \caption{}   
    \label{fig: h_order_2}
\end{figure}

Cohen, Dasbach and Russell explained the Conway polynomial as a dimer model \cite[Theorem~2.6 and Section~3]{CohenDasbachRussell2014}, and showed that $\op{Perm}(G(L'))=\pm\det (M(L'))$, where $M(L')$ is called the {\em Kasteleyn matrix} transformed from $G(L')$ by changing the local weights; see the lower line of Figure \ref{fig: h_order_2}. Here the $\pm$ sign only depends on the planar projection of $L'$. Note that the weights in positive and negative crossings are related by exchanging $x$ and $x^{-1}$.

\noindent 
Step (b): The matrices $M(L_{\epsilon})$ for different crossing lifts differ by changing the row vectors corresponding to positive or negative crossings. So the average $\overline{\nabla}(L)$ is equal to
$$2^{-c}\sum_{\epsilon}\nabla(L_{\epsilon})=2^{-c}\sum_{\epsilon}\op{Perm}(G(L_{\epsilon}))=\pm 2^{-c}\sum_{\epsilon}\det(M(L_{\epsilon}))=\pm\det (M(L)).$$
Here each row of $M(L)$ is the average of the two rows corresponding to the positive and negative crossings. The local weights of $M(L)$ are $r, -1, r, -1$ as in Figure \ref{fig: h_order_3}, where $r=\frac{x+x^{-1}}{2}$. Thus, $M(L)$ is a matrix whose entries are $0, -1, r$.

\begin{figure}[ht]
    \centering
    \includegraphics[width=2.1cm]{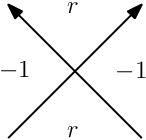}
    \caption{}   
    \label{fig: h_order_3}
\end{figure}

\noindent 
Step (c):  We suppress $L$ in $M(L)$, and write $M$ for simplicity. We write $M=M_1+ (r-1)S$, where $M_1=M|_{r=1}$ and $S$ do not contain $r$. By expanding $\det(M)$ in multiple columns, we see that $\dim\ker (M_1)\ge n-1$ implies that $(r-1)^{n-1}|\det (M)$. Let $\tilde{M}$ denote the $c \times (c+2)$ matrix obtained from $M$ by adding the two deleted columns back. It is enough to show that $\dim \ker (\tilde{M}_1) \ge n+1$. 

Denote the components of $L$ by $K_1, \cdots, K_n$. For each $K_j$ $(1\leq j\leq n)$, we associate an $(c+2)$-dimensional vector $w_j$ by letting $w_j(R)$ be the winding number of $K_j$ around the region $R$. Winding numbers add up to zero for the four regions near each crossing. Thus, $w_j\in\ker (\tilde{M}_1)$. Define $w_0$ by $w_0(R)=1$ for any region $R$. Then $w_0\in\ker(\tilde{M}_1)$ since each row of $\tilde{M}_1$ contains exactly two $1$'s and two $-1$'s.

It remains to show that $w_0, w_1, \dots, w_n$ are linearly independent. Assume $\sum\limits_{j=1}^na_jw_j=0$, since $w_j(R_0)=0$ $(1\leq j\leq n)$ for the unbounded region $R_0$, thus $a_0=0$. Let $R$ be a region adjacent to $R_0$, separated by an arc of $K_j$. Then $R$ lies outside of $K_i$ for any $i\neq j$, so that $w_i(R)=0$. This implies that $a_j=0$. Repeating the above argument, we concludes that all $a_i$ vanishes. Note that here the assumption that $L$ is nonsplittable is essentially used.

We finally show that this estimate of $h$-order is sharp. Recall from Example \ref{ex B} that $L_3$ is the Borromean ring, and $B(L_3)=h^2(1+h)^{-3}$. Take a connected sum of $n$ copies of $L_3$, denoted by $L_3^n$, which has $2n+1$ components. We have $B(L_3^n)=B(L_3)^n=h^{2n}(1+h)^{-3n}$, whose $h$-order is $2n$.
\end{enumerate}
\end{proof}

As an application, we can use the polynomial invariant $B(L)$ to distinguish plane curves with the same invariant $\operatorname{St}$.
\begin{example}
Let $K, K'$ be the plane curves depicted in Figure \ref{fig: 6}.
\begin{figure}[ht]
    \centering
    \includegraphics[width=10cm]{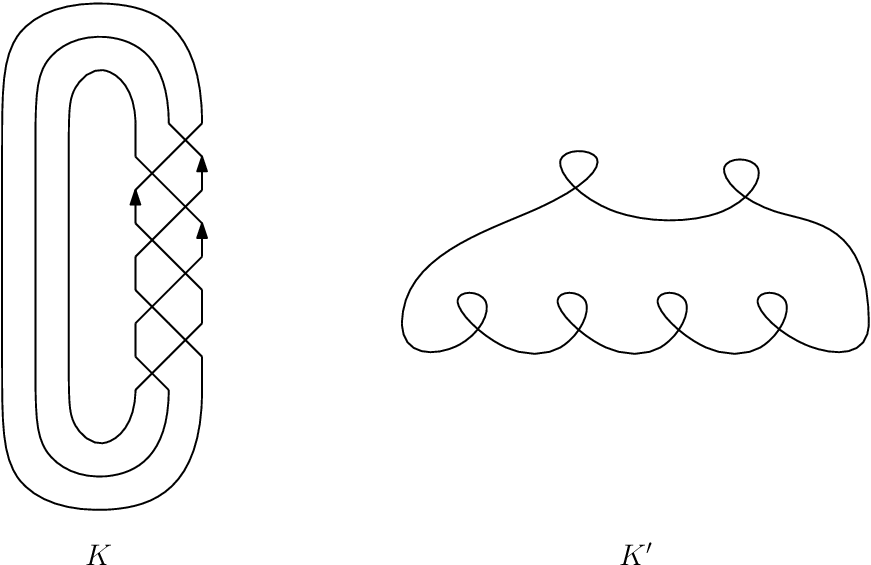}
    \caption{}   
    \label{fig: 6}
\end{figure}
Note that here we omit the orientation of $K'$ due to Proposition \ref{proposition4.1}. Direct calculation shows that 
\begin{center}
$B(K')=(1+h)^{-4}=1-4h+10h^2-\cdots$,
\end{center}
and
\begin{center}
$B(K)=(1-h)^2(1+h)^{-2}=1-4h+8h^2-\cdots$.
\end{center}
Thus they have the same $\op{St}=4$, but different $b_2$.
\end{example}

\printbibliography

\end{document}